\documentclass[11pt,a4paper]{amsart}
\usepackage{hyperref,texmac,enumitem}
\usepackage[all]{xy}
\usepackage{amsmath,mathtools}
\xyoption{web}

\usepackage{bbm}

\operators{Aut,Div,Pic,Pid,Id,Cl,ab,Norm,cts,Ext,opp,J,M,skew,com,Core,Stab,Epi,sign}
\operators{Ara,fin,im,torsfree,c,Nm,tori,D,f,ring,CL,Surj,Av,PSL,cpt}
\calsymbols{c}{O,C,B,A,R,K,M,D,S,F,X,Y,L,P,G,E}
\bbsymbols{b}{E,A,P,T,H}
\fraksymbols{f}{F,p,A,M,m,a,n}

\NewDocumentCommand\FbL{O{b} O{L}}{\cF^{#1}_{#2}}

\begin{document}
\author{Alex Bartel}
\address{School of Mathematics and Statistics, University of Glasgow,
University Place, Glasgow G12 8SQ, United Kingdom}
\author{Adam Morgan}
\address{Gonville and Caius College, Trinity Street, Cambridge, CB2 1TA, United Kingdom}

\email{alex.bartel@glasgow.ac.uk, ajm269@cam.ac.uk}
\date{\today}
\title[Reasonable functions in Cohen--Lenstra heuristics]{On the notion of reasonable functions in Cohen--Lenstra type heuristics}

\begin{abstract}
    We elucidate the notion of ``reasonable functions'' in the context of Cohen--Lenstra heuristics. For a sequence $(x_i)_i$ of independent random variables valued in an infinite discrete probability space $(X,p)$, we investigate for what sets $S$ of functions on $X$ it is true that almost surely, for all $f\in S$, the average of $f$ on the sequence is equal to its expected value. We show that if $p$ decays sufficiently quickly, as e.g. the Cohen--Lenstra distribution does, then $\bigcup_{q>1}L^q(X)$ is such a set; while if $p$ decays sufficiently slowly, then not even $\bigcap_{q<\infty}L^q(X)$ has this property.
\end{abstract}

\maketitle

\section{Introduction}
We are motivated by the following problem: one is given a sequence of certain objects, and one wishes to express the notion that this sequence ``looks like'' the objects were drawn independently from a specific discrete probability space.
The concrete example that we have in mind is that of the Cohen--Lenstra heuristic \cite{CL},
where the sequence in question is a certain sequence $(G_n)_n$ of finite abelian $\ell$-groups, for a fixed prime number $\ell$,
and the probability space consists of isomorphism classes of finite abelian $\ell$-groups with the probability weight of a group $G$ being inversely proportional to the size of its automorphism group $\Aut(G)$. To express the idea that the given sequence can be convincingly modelled as being random,
the approach taken by Cohen and Lenstra in \cite{CL}, and followed by essentially all other authors in this area, is to postulate that for every ``reasonable'' $\C$-valued function $f$ on the space of isomorphism classes of finite abelian $\ell$-groups, its average $\lim_n \tfrac{1}{n}\sum_{i\leq n}f(G_i)$
should be equal to the expected value $\bE(f)$ with respect to the probability distribution $p$. This immediately raises the question: what should be considered a reasonable function?

Of course for $f$ to be reasonable, $\bE(f)$ should exist in the first place, i.e. $f$ should be an $L^1$ function. However, as \cite[Theorem 7.1]{BL20} shows, if the sequence $(G_n)_n$ really were random, then the probability that every $L^1$-function has the right average on that sequence would be $0$. In \cite{BL20} it was conjectured that the statement ``average $=$ expected value'' should hold for all $f$ for which $\bE(|f|^n)$ exists for all $n\geq 1$, but it was not proven that such a statement holds with probability $1$ for a random sequence -- see \cite[Question 7.4]{BL20}. In this short note we answer that question, but also show that the conjecture of \cite{BL20} is probably overly conservative.

Let $(X,p)$ be an infinite discrete probability space such that $p(x)>0$ for all $x\in X$.
Let $(Y,\bP)$ be the probability space $X^{\Z_{\geq 1}}$ with the induced product probability measure.
For $y\in Y$ we write $y_i$ for the $i$-th coordinate of $y$. The following definition is taken from \cite[\S 7]{BL20}.
\begin{definition}
Call a set $S$ of functions $f\colon X\to \C$ \emph{promising} if for all $f\in S$, the sum $\bE(|f|) = \sum_{x\in X}|f(x)|p(x)$ is finite, and there is a subset $Y'$ of $Y$ of measure $1$ such that for all $y\in Y'$ and all $f\in S$ we have
$$
    \lim_{n\to \infty} \tfrac{1}{n}\sum_{i=1}^n f(y_i) = \bE(f).
$$
\end{definition}
For example, the strong law of large numbers implies that every singleton $S=\{f\}$ for which $\bE(|f|)$ is finite is promising.

We will take ``almost all $y$'' to mean ``all $y$ in some measure $1$ subset of $Y$''.
For $q\in \R_{>0}$, denote by $L^q(X)$ the set of functions $f\colon X\to \C$ for which $\bE(|f|^q)$ is finite, and $L^\infty(X)$ is the space of bounded functions $X\to \C$.

Our main result, a combination of Theorems \ref{thm:main} and \ref{thm:notpromising}, is the following.
\begin{theorem}\label{thm:intro_main}
    Let $\delta\in (0,1)$. Then the space $L^{1/\delta}(X)$ is promising if and only if $\sum_{x\in X}p(x)^{1-\delta}$ converges.
\end{theorem}

By \cite[Theorem 7.1]{BL20}, the space $L^1(X)$ is never promising, while \cite[Proposition 7.5]{BL20} shows that the space $L^\infty(X)$ is always promising. The next two results complement those just mentioned.

\begin{corollary}[Corollary \ref{cor:verypromising}]\label{cor:intro_verypromising}
Suppose that for all $s>0$, the sum $\sum_{x\in X}p(x)^{s}$ converges. Then $\bigcup_{q>1}L^q(X)$ is promising. 
\end{corollary}

\begin{proposition}[Proposition \ref{prop:veryunpromising}]\label{prop:intro_veryunpromising}
Suppose that for all $s<1$, the sum $\sum_{x\in X}p(x)^{s}$ diverges. Then $\bigcap_{q>1}L^q(X)$ is not promising.
\end{proposition}

We show in Proposition \ref{prop:CL} that the Cohen--Lenstra probability distribution satisfies the hypotheses of Corollary \ref{cor:intro_verypromising}. Thus, it seems reasonable to conjecture that a function $f$ is reasonable for the purposes of \cite[Conjecture 1.5]{BL20} if there exists some $q>1$ such that $\bE(|f|^q)$ exists.

\begin{acknowledgements}
    Both authors are supported by EPSRC programme grant ``Unifying Arithmetic Statistics'', award reference UKRI4256. After completing the first version of the paper, we  asked LLMs to proof-read it. While Corollary \ref{cor:intro_verypromising} was in the original version, along with a partial converse, the precise form of Theorem \ref{thm:intro_main} was suggested by an LLM, as was the argument sketched in Remark \ref{rmrk:Mourier}.
\end{acknowledgements}

\section{The results}
We first prove the ``if'' direction of Theorem \ref{thm:intro_main}. We will use the following auxiliary result.
\begin{proposition}\label{prop:thintailBC}
    Let $\delta\in (0,1)$, and suppose that $\sum_{x\in X}p(x)^{1-\delta}$ converges. Fix $\epsilon\in \R_{>0}$. For $n\in \Z_{\geq 1}$, let
  \[
    A_n=\{y\in Y: \exists x\in X \textup{ with } \left|\tfrac{1}{n}\#\{i\leq n: y_i=x\}-p(x)\right| \geq \epsilon p(x)^{\delta}\}.
  \]
Then we have $\bP(\{y\in Y: y \in A_n\textup{ for infinitely many $n$}\})=0$.
\end{proposition}
\begin{proof}
    For $x\in X$ and $n\in \Z_{\geq 1}$, let
    $$
      A_{n,x} = \{y\in Y: \left|\tfrac{1}{n}\#\{i\leq n: y_i=x\}-p(x)\right| \geq \epsilon p(x)^{\delta}\},
    $$
    so that we have $A_n = \bigcup_{x\in X}A_{n,x}$.
    For $x\in X$, let
    \begin{eqnarray*}
        A_x & = & \bigcup_{n\in \Z_{\geq 1}} A_{n,x}\\
        & = & \{y\in Y: \exists n\in \Z_{\geq 1} \textup{ with } \left|\tfrac{1}{n}\#\{i\leq n: y_i=x\}-p(x)\right| \geq \epsilon p(x)^{\delta}\}.
    \end{eqnarray*}
    By the Strong Law of Large Numbers, almost all $y\in Y$ have the property that for all $x\in X$, there are only finitely many $n$ for which one has $y\in A_{n,x}$. Thus, it suffices to show that we have
    $$
        \bP(\{y\in Y: y \in A_x\textup{ for infinitely many }x\in X\})=0.
    $$

    Now, for all but finitely many $x\in X$ we have $p(x)< \epsilon p(x)^{\delta}$, 
    so it suffices to restrict attention to such $x$. For those,
    we have $y\in A_{n,x}$ only if $$\tfrac{1}{n}\#\{i\leq n: y_i=x\} > \epsilon p(x)^{\delta}.$$
    Suppose that we have $y\in A_{n,x}$, and say that $$\#\{i\leq n: y_i=x\} = k > n\epsilon p(x)^{\delta}.$$ In particular, we must have $n \leq \left\lfloor\frac{k}{\epsilon p(x)^{\delta}}\right\rfloor$. Call that bound $m_k$. We thus have $A_x \subset \bigcup_{k\in \Z_{\geq 1}}B_k$, where
    $$
    B_k=\{y\in Y: \#\{i\leq m_k: y_i=x\} \geq k\}.
    $$
    Those latter sets satisfy
\begin{eqnarray*}
\bP(B_k) \leq \binom{m_k}{k}p(x)^k
\leq \left(\frac{em_kp(x)}{k}\right)^k
\leq \left(\frac{ep(x)^{1-\delta}}{\epsilon}\right)^k,
\end{eqnarray*}
so that we have
$$
    \bP(A_x) \leq \sum_{k\geq 1}\left(\frac{ep(x)^{1-\delta}}{\epsilon}\right)^k  = \frac{ep(x)^{1-\delta}}{\epsilon}\cdot\left(1+\frac{ep(x)^{1-\delta}}{\epsilon}+\left(\frac{ep(x)^{1-\delta}}{\epsilon}\right)^2+\cdots\right).
$$
By again ignoring finitely many $x\in X$, we may assume that $\frac{ep(x)^{1-\delta}}{\epsilon}\leq \tfrac{1}{2}$, so that 
the bracketed term on the right hand side is bounded above by $2$ for all but finitely many $x\in X$. In summary, we have $\bP(A_x)\ll p(x)^{1-\delta}$. By assumption, the sum $\sum_{x\in X}\bP(A_x)\ll\sum_{x\in X}p(x)^{1-\delta}$ converges, so by the Borel--Cantelli Lemma we have
$$
\bP(\{y\in Y: y \in A_x\textup{ for infinitely many }x\in X\})=0,
$$
as claimed, which implies the result.
\end{proof}

\begin{definition}
    For $y\in Y$  and $n\in \mathbb{Z}_{\geq 1}$, we define the function $h_{n,y}\colon X\to \mathbb{R}_{\geq 0}$ by
\[h_{n,y}(x)=p(x)^{-1} \cdot \tfrac{1}{n}\#\{i\leq n : y_i=x\}.\]
\end{definition}

Note that  for $y\in Y$, $n\in \mathbb{Z}_{\geq 1}$, and   $f\in L^1(X)$, we have 
\begin{eqnarray} \label{exp:form}
  \frac{1}{n}\sum_{i=1}^nf(y_i) &=& \sum_{x\in X}\tfrac{1}{n}\#\{i\leq n : y_i=x\}f(x)=\bE(h_{n,y}\cdot f).
\end{eqnarray}

 \begin{corollary} \label{cor:almost_sure_convergence}
     Let $\delta\in (0,1)$, and suppose that $\sum_{x\in X}p(x)^{1-\delta}$ converges. 
 Set $u=1/(1-\delta)$.  Then, for almost all  $y\in Y$, we have $\bE(|h_{n,y}-1|^u)\rightarrow 0$ as $n\to \infty$.
 \end{corollary}

 \begin{proof}
 By Proposition \ref{prop:thintailBC}, we have that for almost all $y\in Y$, the quantity 
\[B_n=\sup_{x\in X}|h_{n,y}(x)-1|p(x)^{1/u}\]
tends to $0$ as $n$ tends to $\infty$. To apply this, note that
\begin{eqnarray*}
\bE(|h_{n,y}-1|^u)\leq B_n^{u-1}\bE(|h_{n,y}-1|p^{-\delta})\leq B_{n}^{u-1}\big(\bE(h_{n,y}p^{-\delta})+ \bE(p^{-\delta})\big).
\end{eqnarray*}
Since the function $x\mapsto p(x)^{-\delta}$ is in $L^1(X)$ by assumption,  it follows from \eqref{exp:form} and the Strong Law of Large Numbers that the limit $\lim_{n\to \infty}\bE(h_{n,y}p^{-\delta})$ is equal to $\bE(p^{-\delta})$, hence exists, for almost all $y\in Y$.  This, and the fact that $B_n\to 0$ for almost all $y\in Y$, completes the proof.
 \end{proof}

 \begin{remark}\label{rmrk:Mourier}
Instead of deducing Corollary \ref{cor:almost_sure_convergence} from Proposition \ref{prop:thintailBC}, one can deduce it from Mourier's Strong Law for Banach space valued random variables \cite{M}, as we now sketch. 

For $i\in \Z_{\geq 1},$ define the function $g_i\colon Y\to L^u(X)$ by setting
\[g_i(y)(x)=p(x)^{-1}\mathbbm{1}_{y_i=x},\]
noting that $\tfrac{1}{n}\sum_{i=1}^ng_i(y)=h_{n,y}.$
Now $g_1,g_2,...$ is a sequence of i.i.d. random variables valued in the separable Banach space $L^u(X).$ The composition of $g_1$ with the $L^u$-norm $L^u(X)\to \R_{\geq 0}$ sends $y\in Y$ to $p(y_1)^{-\delta}$, hence has expectation 
\[\sum_{x\in X}p(x)^{1-\delta}<\infty.\]
Thus, by the version of Mourier's Strong Law stated as \cite[Corollary 3.7.21]{GN}, there is a measure $1$ subset  $Y_0$ of $Y$ on which the functions $\tfrac{1}{n}\sum_{i=1}^ng_i(y)=h_{n,y}$ converge in the $L^u$-norm to the Bochner expectation $\bE(g_1)$ of $g_1$. One computes that $\bE(g_1)$ is the constant function $1$, so for all $y\in Y_0$ we have 
$\bE(|h_{n,y}-1|^u)^{1/u}\rightarrow 0$ as $n\rightarrow \infty.$
 \end{remark}

\begin{theorem} \label{thm:main}
Fix $\delta\in (0,1)$ and set $q=1/\delta$. Suppose that $\sum_{x\in X}p(x)^{1-\delta}$ converges.
Then $L^q(X)$ is promising. 
\end{theorem}

\begin{proof}

Set $u=1/(1-\delta)$, so that we have $1/u+1/q=1$.  
Take $f\in L^q(X)$.  
By \eqref{exp:form} and H\"{o}lder's inequality, we have 
\begin{eqnarray*}
\Big| \frac{1}{n}\sum_{i=1}^nf(y_i)-\bE(f)\Big|&=&  \left|\bE((h_{n,y}-1)\cdot f)\right|\\
&\leq &   \bE(|h_{n,y}-1|^u)^{1/u}\bE(|f|^q)^{1/q}.
\end{eqnarray*}
By Corollary \ref{cor:almost_sure_convergence}, there is a measure $1$ subset of $Y$, independent of the choice of $f$, on which the right hand side tends to $0$ as $n\to \infty$.  
\end{proof}

\begin{corollary}\label{cor:verypromising}
Suppose that $\sum_{x\in X}p(x)^{s}$ converges for all $s>0$. Then $\bigcup_{q>1}L^q(X)$ is promising. 
\end{corollary}

\begin{proof}
For any fixed $q>1$, the set $L^q(X)$ is promising by Theorem \ref{thm:main}. Now conclude by taking a sequence $(q_n)_n$ tending to $1$ from above.
\end{proof}

We now prove the other direction of Theorem \ref{thm:intro_main}. We follow the strategy of \cite[\S 7]{BL20}. First, we have the following generalisation of \cite[Proposition 7.2]{BL20}.
\begin{lemma}\label{lem:overeager}
    Let $\delta\in (0,1)$ be such that we have $\sum_{x\in X}p(x)^{1-\delta} = \infty$.
    Then there is a subset $Y'$ of $Y$ of measure $1$ with the property that for all $y\in Y'$, for all $\epsilon\in \R_{>0}$, there exist infinitely many $x\in X$ such that for some $i\leq \epsilon p(x)^{-\delta}$ one has $y_i=x$.
\end{lemma}
\begin{proof}
    First, fix $\epsilon\in \R_{>0}$ and $x\in X$. For a random $y\in Y$, the probability that we have $x=y_i$ for some $i\leq \epsilon p(x)^{-\delta}$ is $1-(1-p(x))^{\lfloor \epsilon p(x)^{-\delta}\rfloor}\sim \epsilon p(x)^{1-\delta}$. Note that these events for different $x$ are not independent, but they are negatively correlated, which allows us to apply the Borel--Cantelli Lemma as stated in \cite[p. 48 ff.]{Borel} or \cite[eqn. (0)]{ChungErdos}. Since $\sum_x p(x)^{1-\delta}$ is assumed to diverge, that more general version of the Borel--Cantelli Lemma implies that there is a subset $Y_\epsilon$ of $Y$ of measure $1$ such that for all $y\in Y_\epsilon$ there exist infinitely many $x\in X$ with $x=y_i$ for some $i\leq \epsilon p(x)^{-\delta}$. Now let $(\epsilon_n)_n$ be a sequence of positive real numbers tending to $0$. Then $Y' = \bigcap_n Y_{\epsilon_n}$ satisfies the conclusion of the lemma.
\end{proof}

\begin{theorem}\label{thm:notpromising}
  Let $\delta\in (0,1)$ be such that $\sum_{x\in X}p(x)^{1-\delta}$ diverges. Then for almost all $y\in Y$ there exists a function $f\in L^{1/\delta}(X)$ such that
     $$
     \limsup_n \tfrac{1}{n}\sum_{i=1}^n f(y_i) = \infty.
     $$
\end{theorem}
\begin{proof}
    Let $Y'$ be as in Lemma \ref{lem:overeager}. Then for every $y\in Y'$ there is a sequence $(x_n)_n$ of distinct elements of $X$ such that for each $n\in \Z_{\geq 1}$ we have
    $$\min\{i\in \Z_{\geq 1}: y_i=x_n\} \leq n^{-2\delta-1}p(x_{n})^{-\delta}.$$

For $n\in \Z_{\geq 1}$, let $i(n)=\min\{j\in \Z_{\geq 1}: y_j=x_n\}
\leq n^{-2\delta-1}p(x_{n})^{-\delta}$. For $x\in X$, define $f(x)=0$ if $x\neq x_n$ for any $n\in \Z_{\geq 1}$,
and $f(x_n)=n^{-2\delta}p(x_n)^{-\delta}$. Then we have $\bE(f^{1/\delta})=\sum_{n\in \Z_{\geq 1}}n^{-2}$,
which converges. On the other hand, for every $n\in \Z_{\geq 1}$, one has
$$
\frac{1}{i(n)}\sum_{j=1}^{i(n)}f(y_j)\geq\frac{f(x_n)}{i(n)}\geq \frac{n^{-2\delta}p(x_n)^{-\delta}}{n^{-2\delta-1}p(x_{n})^{-\delta}}=n,
$$
which gets arbitrarily large, as $n$ varies.
\end{proof}

\begin{proposition}\label{prop:veryunpromising}
Suppose that for all $s<1$, the sum $\sum_{x\in X}p(x)^{s}$ diverges. Then $\bigcap_{q>1}L^q(X)$ is not promising.
\end{proposition}
\begin{proof}
    For $n\in \Z_{\geq 1}$, apply Lemma \ref{lem:overeager} with $\delta=1/n$ to obtain a measure $1$ subset $Y_n\subset Y$. Then the intersection $\bar{Y}=\bigcap_{n\in \Z_{\geq 1}}Y_n$ has measure $1$ and has the property that for all $y\in \bar{Y}$ there exists a sequence $(x_n)_n$ of pairwise distinct elements of $X$ with $p(x_n) \leq 2^{-n}$ and $\min(i: y_i = x_n) \leq n^{-2}p(x_n)^{-1/n}$ for all $n\in \Z_{\geq 1}$. Denote that minimum by $i(n)$.

    Define $f\colon X\to \R$ by $f(x)=0$ if $x\neq x_n$ for any $n$, and $f(x_n) = n^{-1}p(x_n)^{-1/n}$ for all $n\in \Z_{\geq 1}$. Then, for all $q>1$, we have
    $$
        \bE(f^q) = \sum_{n=1}^{\infty}n^{-q}p(x_n)^{1-q/n}.
    $$
    For all $n>2q$ we have $p(x_n)^{1-q/n} \leq 2^{-n/2}$, so the sum converges. Thus, we have $f\in \bigcap_{q>1}L^q(X)$. On the other hand, for all $n\in \Z_{\geq 1}$ we have $\tfrac{1}{i(n)}\sum_{j=1}^{i(n)}f(y_j) \geq n$, hence $\limsup_m \tfrac{1}{m}\sum_{i=1}^m f(y_i) = \infty$.
\end{proof}


We conclude by showing that the Cohen--Lenstra probability distribution satisfies the hypotheses of Corollary \ref{cor:verypromising}.
\begin{proposition}\label{prop:CL}
    Fix a prime number $\ell$, and let $X$ be a full set of isomorphism class representatives of finite abelian $\ell$-groups.
    Define a probability distribution $p$ on $X$ by setting, for a finite abelian $\ell$-group $A$,
    $$
        p(A) = \prod_{k=1}^\infty(1-\ell^{-k})\cdot \frac{1}{\#\Aut A}.
    $$
    Then for all $s\in \R_{>0}$, the sum
    $$
        \sum_{A\in X}p(A)^s
    $$
    converges.
\end{proposition}
\begin{proof}
    Ignoring the normalising constant, write the sum as
    $$
        \sum_{n=0}^\infty \sum_{A\in X: \atop\#A = \ell^n}\#(\Aut A)^{-s}. 
    $$
    Let $A$ be a finite abelian $\ell$-group, let $e$ be its exponent. Then we may write $A = \Z/e\Z \times H$
    for an abelian $\ell$-group $H$, and observe that $\Aut A$ contains $(\Z/e\Z)^\times \ltimes \Hom(H,\Z/e\Z)$ as a subgroup. We therefore have $\#\Aut A \geq \tfrac{\ell-1}{\ell}\#A$. Next, for $n\in \Z_{\geq 0}$, the number of abelian groups of order $\ell^n$ is the number of partitions of $n$, which is $<e^{C\sqrt{n}}$ for a constant $C>0$. Thus, the sought-for sum is
    $$
        \sum_{n=0}^\infty \sum_{A\in X:\atop \#A = \ell^n}\#(\Aut A)^{-s}
        \ll \sum_{n=0}^\infty \exp(C\sqrt{n}-ns\log\ell),
   $$
    which converges.
\end{proof}
\begin{remark}
    The most general form of the Cohen--Lenstra--Martinet heuristics, as formulated in \cite[Conjecture 1.5]{BL20}, takes as input a finite group $G$ and a finitely generated $\Q[G]$-module $V$ that is orthogonal to the trivial module $\Q$, and defines a probability distribution $p_{G,V}$ on certain finite $\Z_{\ell}[G]$-modules, where $\ell$ is a ``good'' prime number and $\Z_{\ell}$ denotes the ring of $\ell$-adic integers. The probability distribution discussed in Proposition \ref{prop:CL} is the one corresponding to the choice $G=\Z/2\Z$ and $V=0$. Essentially the same proof carries over to the case of general $G$, using the description of finite modules over local orders in semisimple algebras given in \cite[Lemma 2.7(ii)]{CM}. Moreover, the case of general $V$ follows from the case $V=0$, since for an arbitrary $V$ there is a constant $C$ such that we have $p_{G,V}(A) \leq C p_{G,0}(A)$ for all suitable $\Z_{\ell}[G]$-modules $A$.
\end{remark}

\end{document}